\documentclass[12pt]{article}
\usepackage{amscd}
\usepackage{amsfonts}
\usepackage{amsmath}
\usepackage{amssymb}
\usepackage{amsthm}
\usepackage{bbm}
\usepackage{CJK}
\usepackage{fancyhdr}
\usepackage{graphicx}
\usepackage{indentfirst}
\usepackage{latexsym,bm}
\usepackage{mathrsfs}
\usepackage{xypic}
\usepackage[square, comma, sort&compress, numbers]{natbib}
\allowdisplaybreaks[4]
\newtheorem{theorem}{Theorem}[section]
\newtheorem{lemma}[theorem]{Lemma}

\newtheorem{proposition}[theorem]{Proposition}
\newtheorem{example}[theorem]{Example}

\usepackage[top=1in,bottom=1in,left=1.25in,right=1.25in]{geometry}
\def\<{\langle}
\def\>{\rangle}
\def\a{\alpha}
\def\b{\beta}

\def\g{\gamma}

\def\o{\otimes}

\def\tr{\triangleright}
\def\tl{\triangleleft}

\date{}
\begin{document}
\renewcommand{\baselinestretch}{1.2}
\renewcommand{\arraystretch}{1.0}

\title{\bf The bicrossed products of $H_4$ and $H_8$}

\author {{ Daowei Lu$^{1,2}$, Yan Ning$^2$, Dingguo Wang$^1$\footnote {Corresponding author: Email: dgwang@qfnu.edu.cn.} }\\
{\small $^1$ School of Mathematical Sciences, Qufu Normal University,}\\
{\small Qufu Shandong, 273165, P. R. of China}\\
{\small $^2$ Department of Mathematics, Jining University,}\\
{\small Qufu Shandong, 273155, P. R. of China}\\
}

 \maketitle
\begin{center}
\begin{minipage}{12.cm}

\noindent{\bf Abstract.}
Let $H_4$ and $H_8$ be the Sweedler's and Kac-Paljutkin Hopf algebras, respectively. In this paper we prove that any Hopf algebra which factorizes through $H_8$ and $H_4$ (equivalently, any bicrossed product between the Hopf algebras $H_8$ and $H_4$) must be isomorphic to one of the following four Hopf algebras: $H_8\o H_4,H_{32,1},H_{32,2},H_{32,3}$. The set of all matched pair $(H_8,H_4,\tr,\tl)$ is explicitly described, and then the associated bicrossed products is given by generators and relations.
\\

\noindent{\bf Keywords:}
 Kac-Paljutkin Hopf algebra; Sweedler's Hopf algebra; Bicrossed product; The factorization problem.
\\

 \noindent{\bf  Mathematics Subject Classification:} 16T10,16T05,16S40.
 \end{minipage}
 \end{center}
 \normalsize\vskip1cm

\section{Introduction}\label{sec1}

The factorization problem stemmed from group theory and it was first considered in \cite{Ma} by Maillet. This problem aims at the description
and classification of all groups $G$ which factor through two given groups $N$ and $H$, i.e. $G=NH$, and $N\cap H=\{1\}$. In \cite{Kassel}, the group $X$ is called bicrossed product of $N$ and $G$.  However, although the statement of the problem seems very simple and natural, no major progress has been made so far since we still have not commanded exhaustive methods to handle it. For example, even the description and classification of
groups which factor through two finite cyclic groups is still an open problem.

An important step in dealing with the factorization problem for groups was the bicrossed product construction introduced in the paper \cite{Z} by Zappa; later on, Takeuchi discovered the same construction in the paper \cite{T}, where the terminology bicrossed product was firstly brought up. The main ingredients in this construction are the so-called matched pairs of groups.
Subsequently, S. Majid in \cite{Majid} generalized this notion to the context of Hopf algebras, and considered a more computational approach of the problem.
The present paper is a contribution to the factorization problem for Hopf algebras.

In paper \cite{ABM}, the authors proposed a strategy for classifying bicrossed product of Hopf algebras following the Majid's construction. The method proposed in \cite{ABM} was followed in \cite{B} to classify bicrossed products of two Sweedler's Hopf algebras, in \cite{Kei1,Kei2} to compute the automorphism of Drinfeld doubles of purely non-abelian finite group and quasitriangular structure of the doubles of finite group, respectively; then in \cite{Ago1} to classify bicrossed products of two Taft algebras, and finally in \cite{Ago2} to classify bicrossed products of Taft algebras and group algebras, where the group is a finite cyclic group.

In the 1960's, G.I. Kac and Paljutkin in \cite{KP} discovered a non-commutative, non-cocommutative semisimple Hopf algebra $H_8$ which is 8-dimensional. Later, in \cite{Mas} Masouka showed that there is only one (up to isomorphisms) semisimple Hopf algebra of dimension 8, and presented it under the perspective of biproducts and bicrossed products. Ore extension is an important tool to study Hopf algebras \cite{Pan,WZZ1,WZZ2,WZZ3,XHW}, recently Pansera in \cite{P} constructed $H_8$ from the point of view of Ore extension in the classification of the inner faithful Hopf actions of $H_8$ on the quantum plane.

As we all know, the 4-dimensional Sweedler's Hopf algebra $H_4$ is the simplest non-commutative, non-cocommutative semisimple Hopf algebra. In this paper, we will describe all the bicrossed products between $H_8$ and $H_4$, and prove that any Hopf algebra which factorizes through $H_8$ and $H_4$ must be isomorphic to one of the four Hopf algebras $H_8\o H_4,H_{32,1},H_{32,2},H_{32,3}$ given explicitly in Theorem 5.1.

This paper is organized as follows. In section 1, we will recall the basic definitions and facts needed in our computations. In section 2, we will determine when $H_4$ becomes a left $H_8$-module coalgebra, and in section 3, we will determine when $H_8$ becomes a right $H_4$-module coalgebra. In section 4, we will find the suitable mutual actions between $H_4$ and $H_8$ such that they could make up matched pairs. And the bicrossed products are also given.

Throughout this paper, $k$ will be an arbitrary algebraically closed field of characteristic zero. Unless otherwise
specified, all algebras, coalgebras, bialgebras, Hopf algebras, tensor products and homomorphisms are over $k$.

\section{Preliminaries}
\def\theequation{2.\arabic{equation}}
\setcounter{equation} {0} \hskip\parindent

In this section, we will recall some basic definitions and facts.

Recall that Sweedler's 4-dimensional Hopf algebra, $H_4$, is generated by two elements $G$ and $X$ subjecting to the relations
$$G^2=1,\ \ X^2=0,\ \ XG=-GX.$$
The coalgebra structure and the antipode are given by
\begin{eqnarray*}
&&\Delta(G)=G\o G,\ \ \varepsilon(G)=1,\ \ S(G)=G,\\
&&\Delta(X)=X\o G+1\o X,\ \ \varepsilon(X)=0,\ \ S(X)=GX.
\end{eqnarray*}
It is well known that the set of group-like elements $G(H_4)$ and the set of primitive elements $P_{g,h}(H_4)$ are given as follows
\begin{eqnarray*}
&&G(H_4)=\{1,G\},\ P_{1,1}(H_4)=P_{G,G}(H_4)=\{0\},\\
&&P_{G,1}(H_4)=k(G-1)\oplus kX,\ P_{1,G}(H_4)=k(G-1)\oplus kGX.
\end{eqnarray*}

The Hopf algebra $H_8$ is generated by three elements $g,h$ and $z$ subjecting to the relations
\begin{eqnarray*}
&&g^2=1,\ h^2=1,\ gh=hg,\\
&&z^2=\frac{1}{2}(1+g+h-gh),\ gz=zh,\ hz=zg.
\end{eqnarray*}
The coalgebra structure and antipode are given by
\begin{eqnarray*}
&&\Delta(g)=g\o g,\ \ \varepsilon(g)=1,\ \ S(g)=g,\\
&&\Delta(h)=h\o h,\ \ \varepsilon(h)=1,\ \ S(h)=h,\\
&&\Delta(z)=J(z\o z),\ \varepsilon(z)=1,\ S(z)=z,
\end{eqnarray*}
where $J=\frac{1}{2}(1\o1+g\o 1+1\o h-g\o h).$
Obviously $H_8$ is 8-dimensional with the basis $\{g^ih^jz^k|0\leq i,j,k\leq1, i,j,k\in \mathcal{N}\}$. It is easy to verify that the element $z$ is invertible with $z^4=1$.

A matched pair of Hopf algebras is a quadruple $(A,H,\tr,\tl)$, where $A$ and $H$ are Hopf algebras, $\tr:H\o A\rightarrow A$ and $\tl:H\o A\rightarrow H$ are coalgebra maps such that $A$ is a left $H$-module coalgebra, $H$ is a right $A$-module coalgebra and the following compatible conditions hold:
\begin{align}
&h\tr 1_A=\varepsilon(h)1_A,\ \ 1_H\tl a=\varepsilon(a)1_H,\\
&h\tr(ab)=(h_1\tr a_1)((h_2\tl a_2)\tr b),\\
&(gh)\tl a=(g\tl(h_1\tr a_1))(h_2\tl a_2),\\
&h_1\tl a_1\o h_2\tr a_2=h_2\tl a_2 \o h_1\tr a_1,
\end{align}
for all $a,b\in A,g,h\in H$.

\section{The left $H_8$-module coalgebra structures on $H_4$}
\def\theequation{3.\arabic{equation}}
\setcounter{equation} {0} \hskip\parindent
In this section we will find all the actions $\tr:H_8\o H_4\rightarrow H_4$ such that as a vector space, $H_4$ is made into a left $H_8$-module coalgebra satisfying $x\tr1=\varepsilon(x)1$ for all $x\in H_8$.

First of all, $g\tr1=h\tr1=z\tr1=1.$ Then $g\tr G\in G(H_4)$. We have $g\tr G\neq1$, for otherwise $G=g^2\tr G=1$. Therefore $g\tr G=G$. Similarly $h\tr G=G$.
Since
\begin{align*}
\Delta(z\tr G)&=\frac{1}{2}[z\tr G\o z\tr G+gz\tr G\o z\tr G+z\tr G\o hz\tr G-gz\tr G\o hz\tr G]\\
&=\frac{1}{2}[z\tr G\o z\tr G+zh\tr G\o z\tr G+z\tr G\o zg\tr G-zh\tr G\o zg\tr G]\\
&=z\tr G\o z\tr G.
\end{align*}
Hence $z\tr G\in G(H_4)$. We have $z\tr G\neq1$, for otherwise
$$1=z\tr(z\tr G)=z^2\tr G=\frac{1}{2}(1+g+h-gh)\tr G=G.$$
Therefore $z\tr G=G$.

By computation we have $g\tr X\in P_{G,1}(H_4)$; thus $g\tr X=\a_1(G-1)+\b_1 X$ for any $\a_1,\b_1\in k$. Since the action of $g$ is compatible with $g^2=1$, we have
$$X=g^2\tr X=g\tr(g\tr X)=(1+\b_1)\a_1(G-1)+\b^2_1X.$$
Thus $(1+\b_1)\a_1=0$ and $\b^2_1=1$. Therefore
\begin{equation}\label{A}
g\tr X=\a_1(G-1)- X\ \hbox{or}\ g\tr X=X.
\end{equation}
Similarly for any $\a_2\in k$
\begin{equation}\label{B}
h\tr X=\a_2(G-1)- X\ \hbox{or}\ h\tr X=X.
\end{equation}
Because the actions of $g$ and $h$ on $X$ should be compatible with $gh=hg$, we need to check the actions given in (\ref{A}) and (\ref{B}). For example if $g\tr X=\a_1(G-1)- X,h\tr X=\a_2(G-1)- X$,
$$gh\tr X=g\tr(\a_2(G-1)- X)=\a_2(G-1)-\a_1(G-1)+X,$$
and
$$hg\tr X=h\tr(\a_1(G-1)- X)=\a_1(G-1)-\a_2(G-1)+X.$$
Then $gh=hg$ forces $\a_1=\a_2$. The other three cases could also be checked easily. Now we give the actions of $g,h$ on $G,X$ in the following tables:
 $$\begin{array}{c|ccc}
 \tr^1   &1  &G  &X\\
\hline 1 &1  &G  &X\\
       g &1  &G  &X\\
       h &1  &G  &X\\
      gh &1  &G  &X\\
 \end{array}\ \ \ \
\begin{array}{c|ccc}
 \tr^2   &1  &G  &X\\
\hline 1 &1  &G  &X\\
       g &1  &G  &X\\
       h &1  &G  &\a(G-1)-X\\
      gh &1  &G  &\a(G-1)-X\\
 \end{array}$$

$$\begin{array}{c|ccc}
 \tr^3   &1  &G  &X\\
\hline 1 &1  &G  &X\\
       g &1  &G  &\a(G-1)-X\\
       h &1  &G  &X\\
      gh &1  &G  &\a(G-1)-X\\
 \end{array}\ \ \ \
\begin{array}{c|ccc}
 \tr^4   &1  &G  &X\\
\hline 1 &1  &G  &X\\
       g &1  &G  &\a(G-1)-X\\
       h &1  &G  &\a(G-1)-X\\
      gh &1  &G  &X\\
 \end{array}$$
For the action $z\tr X$, we have
\begin{align*}
\Delta(z\tr X)=&\frac{1}{2}[z\tr X\o G+1\o z\tr X+gz\tr X\o G+1\o z\tr X\\
&+z\tr X\o G+1\o hz\tr X-gz\tr X\o G-1\o hz\tr X]\\
=&z\tr X\o G+1\o z\tr X,
\end{align*}
hence $z\tr X\in P_{G,1}(H_4)$, and $z\tr X=\b(G-1)+\g X$ for $\b,\g\in k$. On one hand,
$$z^2\tr X=z\tr(\b(G-1)+\g X)=\b(1+\g)(G-1)+\g^2 X.$$
On the other hand,
$$z^2\tr X=\frac{1}{2}(1+g+h-gh)\tr X.$$
Therefore the element $z\tr X$ is determined by the actions of $g,h$ on $X$, and we will consider every case.

For the action $\tr^1$, we obtain $\b(1+\g)(G-1)+\g^2 X=X.$ Then
$$z\tr X=X,\ \ \hbox{or}\ \ \ z\tr X=\b(G-1)-X.$$
Straightforward to see that both actions are compatible with the relation
\begin{equation}\label{C}
gz=zh,hz=zg.
\end{equation}

For the action $\tr^2$, we have
$$z\tr X=X,\ \ \hbox{or}\ \ \ z\tr X=\b(G-1)-X.$$
However both actions are not compatible with the relation (\ref{C}). Similarly the action $\tr^3$ does not hold neither.

For the action $\tr^4$, we have
$$\a(G-1)-X=\b(1+\g)(G-1)+\g^2X.$$
Thus $\g=i,\a=(1+i)\b$ or $\g=-i,\a=(1-i)\b$. Since
$$gz\tr X=i[\b(G-1)\pm X]=zh\tr X,$$
the action $\tr^4$ is well defined. So till now we have four actions, which is redenoted by $\tr^1$ and $\tr^2$, as follows
$$\begin{array}{c|ccc}
 \tr^1   &1  &G  &X\\
\hline 1 &1  &G  &X\\
       g &1  &G  &X\\
       h &1  &G  &X\\
       z &1  &G  &X\\
 \end{array}\ \ \ \
 \begin{array}{c|ccc}
 \tr^2   &1  &G  &X\\
\hline 1 &1  &G  &X\\
       g &1  &G  &X\\
       h &1  &G  &X\\
       z &1  &G  &\a(G-1)-X\\
      \end{array}
\begin{array}{c|ccc}
 \tr^3   &1  &G  &X\\
\hline 1 &1  &G  &X\\
       g &1  &G  &\a(G-1)-X\\
       h &1  &G  &\a(G-1)-X\\
       z &1  &G  &\frac{\a}{1+i}(G-1)+iX\\
 \end{array}\ \ \ \
\begin{array}{c|ccc}
 \tr^4   &1  &G  &X\\
\hline 1 &1  &G  &X\\
       g &1  &G  &\a(G-1)-X\\
       h &1  &G  &\a(G-1)-X\\
       z &1  &G  &\frac{\a}{1-i}(G-1)-iX\\
 \end{array}\ \ \ \
$$
Similarly we have the action of $g,h$ and $z$ on $GX$ as follows:
$$\begin{array}{c|c}
 \tr^a     &GX\\
\hline 1   &GX\\
       g   &GX\\
       h  &GX\\
       z   &GX\\
 \end{array}\ \ \ \
 \begin{array}{c|c}
 \tr^{b}     &GX\\
\hline 1   &GX\\
       g   &GX\\
       h   &GX\\
       z  &\b(G-1)-GX\\
      \end{array}\ \ \ \
\begin{array}{c|c}
 \tr^{c}     &GX\\
\hline 1   &GX\\
       g   &\b(G-1)-GX\\
       h   &\b(G-1)-GX\\
       z  &\frac{\b}{1+i}(G-1)+iGX\\
 \end{array}\ \ \ \
\begin{array}{c|c}
 \tr^{d}     &GX\\
\hline 1   &GX\\
       g   &\b(G-1)-GX\\
       h   &\b(G-1)-GX\\
       z   &\frac{\b}{1-i}(G-1)-iGX\\
 \end{array}\ \ \ \
$$

\begin{proposition}
There are 16 kinds of actions of $H_8$ on $H_4$ defined as above.
\end{proposition}

\section{The right $H_4$-module coalgebra structures on $H_8$}
\def\theequation{4.\arabic{equation}}
\setcounter{equation} {0} \hskip\parindent

In this section we will find all the actions $\tl:H_8\o H_4\rightarrow H_8$ such that as a vector space, $H_8$ is made into a right $H_4$-module coalgebra satisfying $1\tl x=\varepsilon(x)1$ for all $x\in H_4$.

\begin{lemma}
\begin{itemize}
  \item [(1)]  $G(H_8)=\{1,g,h,gh\}$.
  \item [(2)]  $P_{g^ih^j,g^mh^n}(H_8)=\a_{ijmn}(g^ih^j-g^mh^n)$ for $0\leq i,j,m,n\leq1$ and $\a_{ijmn}\in k$.
\end{itemize}

\end{lemma}

\begin{proof}
(1) Suppose that $x=\sum^1_{i,j=0}f_{ij}g^ih^j+\sum^1_{i,j=0}e_{ij}g^ih^jz$ is a group-like element of $H_8$. Then by $\Delta(x)=x\o x$, we have
\begin{align*}
&\sum^1_{i,j,m,n=0}f_{ij}f_{mn}g^ih^j\o g^mh^n+\sum^1_{i,j,m,n=0}e_{ij}f_{mn}g^ih^jz\o g^mh^n\\
&+\sum^1_{i,j,m,n=0}e_{mn}f_{ij}g^ih^j\o g^mh^nz+\sum^1_{i,j,m,n=0}e_{ij}e_{mn}g^ih^jz\o g^mh^nz\\
&=\sum^1_{i,j=0}f_{ij}g^ih^j\o g^ih^j+\sum^1_{i,j=0}e_{ij}(g^ih^j\o g^ih^j)J(z\o z).
\end{align*}
By comparison, we obtain that
$$\sum^1_{i,j,m,n=0}e_{ij}f_{m,n}g^ih^jz\o g^mh^n=0,$$
and
$$\sum^1_{i,j,m,n=0}e_{m,n}f_{ij}g^ih^j\o g^mh^nz=0.$$
Hence for all $0\leq i,j,m,n\leq1$, $e_{ij}f_{m,n}=0$. And
\begin{align*}
&f^2_{00}=f_{00},\ f^2_{10}=f_{10},\ f^2_{01}=f_{01},\ f^2_{11}=f_{11},\\
&f_{00}f_{01}=0,\ f_{00}f_{10}=0,\ f_{00}f_{11}=0,\ f_{10}f_{01}=0,\ f_{01}f_{11}=0,\ f_{10}f_{11}=0.
\end{align*}
From the above relations, we can see that $f_{ij}=0$ or $f_{ij}=1$, and if for some pair $i,j$ such that $f_{ij}=1$, the rest must be 0.

If $f_{ij}=0$ for $0\leq i,j\leq1$, since $z$ is invertible, we have
$$\sum^1_{i,j,m,n=0}e_{ij}e_{mn}g^ih^j\o g^mh^n=\sum^1_{i,j=0}e_{ij}(g^ih^j\o g^ih^j)J.$$
By comparison of the coefficients, we get the relations
\begin{align*}
&e_{00}=2e^2_{00},\ e_{00}=2e_{00}e_{10},\ e_{00}=2e_{00}e_{01},\ e_{00}=-2e_{10}e_{01},\\
&e_{10}=2e^2_{10},\ e_{10}=2e_{00}e_{10},\ e_{10}=2e_{10}e_{11},\ e_{10}=-2e_{00}e_{11},\\
&e_{01}=2e^2_{01},\ e_{01}=2e_{01}e_{11},\ e_{01}=2e_{01}e_{00},\ e_{01}=-2e_{11}e_{00},\\
&e_{11}=2e^2_{11},\ e_{11}=2e_{11}e_{01},\ e_{11}=2e_{11}e_{10},\ e_{11}=-2e_{10}e_{01}.
\end{align*}
We claim that $e_{00}=0$, for otherwise $e_{10}=e_{00}=e_{01}=\frac{1}{2}$, which contradicts $e_{00}=-2e_{10}e_{01}$. In the same manner, we obtain that $e_{01}=e_{10}=e_{11}=0$.

(2) Suppose that $x=\sum^1_{i,j=0}f_{ij}g^ih^j+\sum^1_{i,j=0}e_{ij}g^ih^jz$ is a $(g^sh^t,g^mh^n)$-primitive element of $H_8$. We have
\begin{align*}
&\sum^1_{i,j=0}f_{ij}g^ih^j\o g^sh^t+\sum^1_{i,j=0}f_{ij}g^mh^n\o g^ih^j+\sum^1_{i,j=0}e_{ij}g^ih^jz\o g^sh^t+\sum^1_{i,j=0}e_{ij}g^mh^n\o g^ih^jz\\
&=\sum^1_{i,j=0}f_{ij}g^ih^j\o g^ih^j+\sum^1_{i,j=0}e_{ij}(g^ih^j\o g^ih^j)J(z\o z).
\end{align*}
By comparison of the coefficients, $e_{ij}=0(0\leq i,j\leq 1)$, and $f_{ij}=0$ except $f_{st},f_{mn}$ with $f_{st}=-f_{mn}$. Therefore $x=f_{st}(g^sh^t-g^mh^n)$.
The proof is completed.
\end{proof}

Now we will analyze the actions of $G$ on the basis $(1,g,h,gh,z,gz,hz,ghz)$.
It is obvious that elements $g\tl G,h\tl G$ and $gh\tl G$ are all group-like elements of $H_8$. Firstly $g\tl G\neq1$, for otherwise $1=g\tl G^2=g$. Also $h\tl G\neq1,gh\tl G\neq1$. Now we will list all possible cases:
\begin{itemize}
  \item If $g\tl G=g$, then $h\tl G\neq g,$ for otherwise $h=h\tl G^2=g\tl G=g$. On one hand assume that $h\tl G=h$, since $(gh\tl G)\tl G=gh$, we obtain $gh\tl G=gh$. On the other hand, if $h\tl G=gh$, then $gh\tl G=h\tl G^2=h$.
  \item If $g\tl G=h$, then $h\tl G=g$ and $gh\tl G=gh$.
  \item If $g\tl G=gh$, then $gh\tl G=g$ and $h\tl G=h.$
\end{itemize}
Assume that
$$z\tl G=\sum^1_{i,j=0}f_{ij}g^ih^j+\sum^1_{i,j=0}e_{ij}g^ih^jz,$$
for $f_{ij},e_{ij}\in k,0\leq i,j\leq1$. Then we have
$$z=\sum^1_{i,j=0}(f_{ij}g^ih^j\tl G)+\sum^1_{i,j=0}(e_{ij}g^ih^jz\tl G).$$
Clearly $f_{ij}=0(0\leq i,j\leq1)$, and $z$ could be linearly represented by $z\tl G,gz\tl G,hz\tl G,ghz\tl G$ which are linearly independent. So are $gz,hz,ghz$. Therefore there exists an invertible matrix $A=(a_{ij})=(\a_1,\a_2,\a_3,\a_4)$ of order 4 such that
$$(z\tl G,gz\tl G,hz\tl G,ghz\tl G)=(z,gz,hz,ghz)A.$$
Since $G^2=1$, we get $A^2=E$. For $z\tl G=a_{11}z+a_{21}gz+a_{31}hz+a_{41}ghz$, on one hand
\begin{align*}
&\Delta(z\tl G)\\
&=a_{11}J(z\o z)+a_{21}(g\o g)J(z\o z)+a_{31}(h\o h)J(z\o z)+a_{41}(gh\o gh)J(z\o z)\\
&=\frac{a_{11}}{2}(z\o z+gz\o z+z\o hz-gz\o hz)\\
&+\frac{a_{21}}{2}(gz\o gz+z\o gz+gz\o ghz-z\o ghz)\\
&+\frac{a_{31}}{2}(hz\o hz+ghz\o hz+hz\o z-ghz\o z)\\
&+\frac{a_{41}}{2}(ghz\o ghz+hz\o ghz+ghz\o gz-hz\o gz)\\
=&[(z,gz,hz,ghz)\o(z,gz,hz,ghz)]\cdot\\
&\frac{1}{2}(a_{11},a_{21},a_{11},-a_{21},a_{11},a_{21},-a_{11},a_{21},a_{31},-a_{41},a_{31},a_{41},-a_{31},a_{41},a_{31},a_{41})^T.
\end{align*}
On the other hand
\begin{align*}
&\Delta(z\tl G)\\
&=\frac{1}{2}(z\tl G\o z\tl G+gz\tl G\o z\tl G+z\tl G\o hz\tl G-gz\tl G\o hz\tl G)\\
&=[(z,gz,hz,ghz)\o(z,gz,hz,ghz)]\frac{1}{2}[\a_1\o \a_1+\a_2\o\a_1+\a_1\o\a_3-\a_2\o\a_3].
\end{align*}
Hence we have the relation
\begin{align}\label{D}
&(a_{11},a_{21},a_{11},-a_{21},a_{11},a_{21},-a_{11},a_{21},a_{31},-a_{41},a_{31},a_{41},-a_{31},a_{41},a_{31},a_{41})^T\nonumber\\
&=\a_1\o \a_1+\a_2\o\a_1+\a_1\o\a_3-\a_2\o\a_3.
\end{align}
Similarly for $gz\tl G,hz\tl G,ghz\tl G,$ we have the relations
\begin{align}
&(a_{12},a_{22},a_{12},-a_{22},a_{12},a_{22},-a_{12},a_{22},a_{32},-a_{42},a_{32},a_{42},-a_{32},a_{42},a_{32},a_{42})^T\nonumber\\
&=\a_2\o \a_2+\a_1\o\a_2+\a_2\o\a_4-\a_1\o\a_4,\label{E}\\
&(a_{13},a_{23},a_{13},-a_{23},a_{13},a_{23},-a_{13},a_{23},a_{33},-a_{43},a_{33},a_{43},-a_{33},a_{43},a_{33},a_{43})^T\nonumber\\
&=\a_3\o \a_3+\a_4\o\a_3+\a_3\o\a_1-\a_4\o\a_1,\label{F}\\
&(a_{14},a_{24},a_{14},-a_{24},a_{14},a_{24},-a_{14},a_{24},a_{34},-a_{44},a_{34},a_{44},-a_{34},a_{44},a_{34},a_{44})^T\nonumber\\
&=\a_4\o \a_4+\a_3\o\a_4+\a_4\o\a_2-\a_3\o\a_2\label{G}.
\end{align}

\begin{lemma}
In order to make $H_8$ be a right $H_4$ module coalgebra, the matrix $A$ associated to the actions of $G$ on the basis $(z,gz,hz,ghz)$ satisfies the conditions $A^2=E$ and the identities (\ref{D})--(\ref{G}).
\end{lemma}

\begin{example}
(1) Let $A=E$, then clearly $A$ satisfies the relations (\ref{D})-(\ref{G}).

(2) Let $A$ be the matrix
\begin{equation*}
\left(
  \begin{array}{cccc}

    0 & 0 & 0 & 1\\
    0 & 0 & 1 & 0\\
    0 & 1 & 0 & 0\\
    1 & 0 & 0 & 0
  \end{array}
\right).
\end{equation*}
That is, $z\tl G=ghz,\ gz\tl G=hz,\ hz\tl G=gz,\ ghz\tl G=z$. By a long and tedious verification, $A^2=E$ and satisfies the relations  (\ref{D})--(\ref{G}).
\end{example}

Next we will consider the actions of $X$ on the basis of $H_8$. Since
\begin{align*}
&\Delta(g\tl X)=g\tl X\o g\tl G+g\o g\tl X,\\
&\Delta(h\tl X)=h\tl X\o h\tl G+h\o h\tl X,\\
&\Delta(gh\tl X)=gh\tl X\o gh\tl G+gh\o gh\tl X,
\end{align*}
we need to consider all the cases for the actions of $G$ on $g,h,gh$.
\begin{itemize}
  \item When $g\tl G=g,h\tl G=h,gh\tl G=gh$, then $g\tl X=h\tl X=gh\tl X=0$.
  \item When $g\tl G=g,h\tl G=gh,gh\tl G=h$, then $g\tl X=0$, and
  $$h\tl X=\a(h-gh),\ gh\tl X=\b(h-gh).$$
  Moreover $GX=-XG,X^2=0$ implies $\a=\b$.
  \item When $g\tl G=h,h\tl G=g,gh\tl G=gh$, then $gh\tl X=0$, and
  $$g\tl X=\a(g-h),\ h\tl X=\a(g-h).$$
  \item When $g\tl G=gh,h\tl G=h,gh\tl G=g$, then $h\tl X=0$, and
  $$g\tl X=\a(g-gh),\ gh\tl X=\a(g-gh).$$
\end{itemize}

Suppose that $z\tl X=\sum^1_{i,j=0}f_{ij}g^ih^j+\sum^1_{i,j=0}e_{ij}g^ih^jz$.
\begin{align*}
&\sum^1_{i,j,m,n=0}f_{ij}f_{mn}g^ih^j\o g^mh^n+\sum^1_{i,j,m,n=0}e_{ij}f_{mn}g^ih^jz\o g^mh^n\\
&+\sum^1_{i,j,m,n=0}e_{mn}f_{ij}g^ih^j\o g^mh^nz+\sum^1_{i,j,m,n=0}e_{ij}e_{mn}g^ih^jz\o g^mh^nz\\
&=\frac{1}{2}[z\tl X\o z\tl G+gz\tl X\o z\tl G+z\tl X\o hz\tl G-gz\tl X\o hz\tl G\\
&+z\o z\tl X+gz\o z\tl X+z\o hz\tl X-gz\o hz\tl X].
\end{align*}
Because the items like $g^ih^j\o g^mh^n$ will not appear in the right side, $f_{ij}=0$ for all $0\leq i,j\leq1$. Thus $z\tl X$ is a linear combination of $z,gz,hz,ghz$; so are $gz\tl X,hz\tl X,ghz\tl X$. Therefore there exists a matrix $B=(\b_1,\b_2,\b_3,\b_4)$ of order 4 such that
$$(z\tl X,gz\tl X,hz\tl X,ghz\tl X)=(z,gz,hz,ghz)B.$$
The relations $GX=-XG,X^2=0$ implies $AB=-BA,B^2=0$ respectively.

Now we will consider $\Delta(z\tl X),\Delta(gz\tl X),\Delta(hz\tl X)$ and $\Delta(ghz\tl X)$. On one hand,
\begin{align*}
&\Delta(z\tl X)\\
&=b_{11}J(z\o z)+b_{21}(g\o g)J(z\o z)+b_{31}(h\o h)J(z\o z)+b_{41}(gh\o gh)J(z\o z)\\
&=\frac{b_{11}}{2}(z\o z+gz\o z+z\o hz-gz\o hz)\\
&+\frac{b_{21}}{2}(gz\o gz+z\o gz+gz\o ghz-z\o ghz)\\
&+\frac{b_{31}}{2}(hz\o hz+ghz\o hz+hz\o z-ghz\o z)\\
&+\frac{b_{41}}{2}(ghz\o ghz+hz\o ghz+ghz\o gz-hz\o gz)\\
=&[(z,gz,hz,ghz)\o(z,gz,hz,ghz)]\cdot\\
&\frac{1}{2}(b_{11},b_{21},b_{11},-b_{21},b_{11},b_{21},-b_{11},b_{21},b_{31},-b_{41},b_{31},b_{41},-b_{31},b_{41},b_{31},b_{41})^T.
\end{align*}
On the other hand,
\begin{align*}
&\Delta(z\tl X)\\
&=\frac{1}{2}[z\tl X\o z\tl G+gz\tl X\o z\tl G+z\tl X\o hz\tl G-gz\tl X\o hz\tl G\\
&+z\o z\tl X+gz\o z\tl X+z\o hz\tl X-gz\o hz\tl X]\\
&=[(z,gz,hz,ghz)\o(z,gz,hz,ghz)]\cdot\\
&\frac{1}{2}[\b_1\o\a_1+\b_2\o\a_1+\b_1\o\a_3-\b_2\o\a_3+e_1\o\b_1+e_2\o\b_1+e_1\o\b_3-e_2\o\b_3].
\end{align*}
Therefore
\begin{align}
&(b_{11},b_{21},b_{11},-b_{21},b_{11},b_{21},-b_{11},b_{21},b_{31},-b_{41},b_{31},b_{41},-b_{31},b_{41},b_{31},b_{41})^T\nonumber\\
&=\b_1\o\a_1+\b_2\o\a_1+\b_1\o\a_3-\b_2\o\a_3+e_1\o\b_1+e_2\o\b_1+e_1\o\b_3-e_2\o\b_3\label{H},\\
&(b_{12},b_{22},b_{12},-b_{22},b_{12},b_{22},-b_{12},b_{22},b_{32},-b_{42},b_{32},b_{42},-b_{32},b_{42},b_{32},b_{42})^T\nonumber\\
&=\b_2\o \a_2+\b_1\o\a_2+\b_2\o\a_4-\b_1\o\a_4+e_2\o\b_2+e_1\o\b_2+e_2\o\b_4-e_1\o\b_4,\label{I}\\
&(b_{13},b_{23},b_{13},-b_{23},b_{13},b_{23},-b_{13},b_{23},b_{33},-b_{43},b_{33},b_{43},-b_{33},b_{43},b_{33},b_{43})^T\nonumber\\
&=\b_3\o \a_3+\b_4\o\a_3+\b_3\o\a_1-\b_4\o\a_1+e_3\o\b_3+e_4\o\b_3+e_3\o\b_1-e_4\o\b_1,\label{J}\\
&(b_{14},b_{24},b_{14},-b_{24},b_{14},b_{24},-b_{14},b_{24},b_{34},-b_{44},b_{34},b_{44},-b_{34},b_{44},b_{34},b_{44})^T\nonumber\\
&=\b_4\o \a_4+\b_3\o\a_4+\b_4\o\a_2-\b_3\o\a_2+e_4\o\b_4+e_3\o\b_4+e_4\o\b_2-e_3\o\b_2.\label{K}
\end{align}

\begin{lemma}
In order to make $H_8$ be a right $H_4$ module coalgebra, the matrix $B$ associated to the actions of $X$ on the basis $(z,gz,hz,ghz)$ satisfies the conditions $B^2=0$ and the identities (\ref{H})--(\ref{K}).
\end{lemma}

\section{The matched pairs between $H_8$ and $H_4$}
\def\theequation{5.\arabic{equation}}
\setcounter{equation} {0} \hskip\parindent

In this section, we will find the suitable actions which make $H_8$ and $H_4$ matched pairs.
We firstly check the relation $(2.4)$ for the pairs $(g,X),(h,X),(gh,X)$, which should satisfy the identities
\begin{align*}
&g\tl X\o G+g\o g\tr X=g\tl X \o 1 + g\tl G\o g\tr X,\\
&h\tl X\o G+h\o h\tr X=h\tl X \o 1 + h\tl G\o h\tr X,\\
&gh\tl X\o G+gh\o gh\tr X=gh\tl X \o 1 + gh\tl G\o gh\tr X.
\end{align*}
When only $g\tl G=g,h\tl G=h,gh\tl G=gh$, the above identities hold. At this moment, $g\tl X=h\tl X=gh\tl X=0$.

\begin{align*}
1=&z\tr G^2=\frac{1}{2}[(z\tr G)((z\tl G)\tr G)+(gz\tr G)((z\tl G)\tr G)\\
&+(z\tr G)((hz\tl G)\tr G)-(gz\tr G)((hz\tl G)\tr G)]\\
&=G((z\tl G)\tr G)\\
&=a_{11}+a_{21}+a_{31}+a_{41}.
\end{align*}

For the pairs $(g^ih^j,G),(g^ih^jz,G)$ the relation (2.4) is trivial.

On one hand
$$z^2\tl G=\frac{1}{2}(1+g+h-gh)\tl G=z^2.$$
On the other hand by relation (2.3)
\begin{align*}
z^2\tl G&=\frac{1}{2}[(z\tl (z\tr G))(z\tl G)+(z\tl (gz\tr G))(z\tl G)\\
&+(z\tl (z\tr G))(hz\tl G)-(z\tl (gz\tr G))(hz\tl G)]\\
&=(z\tl  G)(z\tl G)\\
&=(a_{11}z+a_{21}gz+a_{31}hz+a_{41}ghz)^2.
\end{align*}
By the comparison of coefficients, we have
\begin{equation}\label{L}
\left\{
\begin{aligned}
&a^2_{11}+2a_{21}a_{31}+a^2_{41}=1,\\
&(a_{11}+a_{41})(a_{21}+a_{31})=0,\\
&a^2_{21}+2a_{11}a_{41}+a^2_{31}=0.
\end{aligned}
\right.
\end{equation}

\begin{align*}
gz\tl G&=(g\tl G)(z\tl G)=a_{11}gz+a_{21}z+a_{31}ghz+a_{41}hz\\
&=(z\tl G)h=zh\tl G,
\end{align*}
which is compatible with $gz=zh$.
Thus
$$a_{11}gz+a_{21}z+a_{31}ghz+a_{41}hz=a_{12}z+a_{22}gz+a_{32}hz+a_{42}ghz.$$
Similarly, we have
\begin{itemize}
  \item $hz\tl G=zg\tl G$, $a_{11}hz+a_{21}ghz+a_{31}z+a_{41}gz=a_{13}z+a_{23}gz+a_{33}hz+a_{43}ghz$.
  \item $ghz\tl G=zgh\tl G$, $a_{11}ghz+a_{21}hz+a_{31}gz+a_{41}z=a_{14}z+a_{24}gz+a_{34}hz+a_{44}ghz$.
\end{itemize}
So we get
\begin{align*}
&a_{11}=a_{22}=a_{33}=a_{44},\ a_{21}=a_{12}=a_{43}=a_{34},\\
&a_{31}=a_{42}=a_{13}=a_{24},\ a_{41}=a_{32}=a_{23}=a_{14}.
\end{align*}
Therefore the matrix $A$ has the following form
\begin{equation*}
\left(
  \begin{array}{cccc}

    a & b & c & d\\
    b & a & d & c\\
    c & d & a & b\\
    d & c & b & a
  \end{array}
\right),
\end{equation*}
where $a,b,c,d\in k$. Since $A^2=E$, plus the relation (\ref{L}), we have the following equation set
\begin{equation*}
\left\{
  \begin{aligned}
&a+b+c+d=1,\\
&a^2+b^2+c^2+d^2=1,\\
&ac+bd=0,\\
&ab+cd=0,\\
&ad+bc=0,\\
&a^2+2bc+d^2=1,\\
&(a+d)(b+c)=0,\\
&b^2+2ad+c^2=0.
  \end{aligned}
\right.
\end{equation*}
We obtain four solutions for the above equation set:
\begin{align*}
&(1)\ a=\frac{1}{2},\ b=\frac{1}{2},\ c=\frac{1}{2},\ d=-\frac{1}{2}.\\
&(2)\ a=-\frac{1}{2},\ b=\frac{1}{2},\ c=\frac{1}{2},\ d=\frac{1}{2}.\\
&(3)\ a=1,\ b=c=d=0.\\
&(4)\ a=b=c=0,\ d=1.
\end{align*}

Now we will verify the actions $g^ih^j\tr GX$ and $g^ih^jz\tr GX$.

For the action $\tr^1$
$$
g\tr GX=(g\tr G)((g\tl G)\tr X)=GX,
$$
and
$$
g\tr XG=(g\tr X)((g\tl G)\tr G)+(g\tr 1)((g\tl X)\tr G)=XG,
$$
which is compatible with the relation $XG=-GX$. Similarly $h\tr GX=gh\tr GX=GX$.

\begin{align*}
z\tr GX=&\frac{1}{2}[(z\tr G)((z\tl G)\tr X)+(gz\tr G)((z\tl G)\tr X)\\
&+(z\tr G)((hz\tl G)\tr X)-(gz\tr G)((hz\tl G)\tr X)]\\
=&\frac{1}{2}G[(z\tl G)\tr X+(z\tl G)\tr X
+(hz\tl G)\tr X-(hz\tl G)\tr X]\\
=&G((z\tl G)\tr X)\\
=&(a_{11}+a_{21}+a_{31}+a_{41})GX\\
=&GX,
\end{align*}
and
\begin{align*}
&z\tr XG\\
&=\frac{1}{2}[(z\tr X)((z\tl G)\tr G)+((z\tl X)\tr G)+(gz\tr X)((z\tl G)\tr G)+((z\tl X)\tr G)\\
&+(z\tr X)((hz\tl G)\tr G)+((hz\tl X)\tr G)-(gz\tr X)((hz\tl G)\tr G)-((hz\tl X)\tr G)]\\
&=X((z\tl G)\tr G)+(z\tl X)\tr G\\
&=(a_{11}+a_{21}+a_{31}+a_{41})XG+(b_{11}+b_{21}+b_{31}+b_{41})G\\
&=XG+(b_{11}+b_{21}+b_{31}+b_{41})G.
\end{align*}
Hence we obtain $b_{11}+b_{21}+b_{31}+b_{41}=0$. That is, $z\tr GX=GX$.
$$
g^ih^j\tr X^2=(g^ih^j\tr X)((g^ih^j\tl G)\tr X)+(g^ih^j\tr 1)((g^ih^j\tl X)\tr X)=X^2=0,
$$

\begin{align*}
&z\tr X^2\\
&=\frac{1}{2}[(z\tr X)((z\tl G)\tr X)+((z\tl X)\tr X)+(gz\tr X)((z\tl G)\tr X)+((z\tl X)\tr X)\\
&+(z\tr X)((hz\tl G)\tr X)+((hz\tl X)\tr X)-(gz\tr X)((hz\tl G)\tr X)-((hz\tl X)\tr X)]\\
&=\frac{1}{2}[X((z\tl G)\tr X)+(z\tl X)\tr X+X((z\tl G)\tr X)+(z\tl X)\tr X\\
&+X((hz\tl G)\tr X)+((hz\tl X)\tr X)-X((hz\tl G)\tr X)-((hz\tl X)\tr X)]\\
&=X((z\tl G)\tr X)+(z\tl X)\tr X\\
&=(b_{11}+b_{21}+b_{31}+b_{41})X\\
&=0,
\end{align*}
which naturally holds.

Since $z^2\tl X=\frac{1}{2}(1+g+h-gh)\tl X=0$,
\begin{align*}
z^2\tl X&=(z\tl X)(z\tl G)+z(z\tl X)\\
&=(b_{11}z+b_{21}gz+b_{31}hz+b_{41}ghz)(az+bgz+chz+dghz)\\
&+z(b_{11}z+b_{21}gz+b_{31}hz+b_{41}ghz)\\
&=0.
\end{align*}
When the solution of matrix $A$ is (1), by the comparison of coefficients, we have
\begin{equation}\label{M}
\left\{
  \begin{aligned}
&2b_{11}+b_{21}+b_{31}=0,\\
&b_{11}=b_{41}.
  \end{aligned}
\right.
\end{equation}

$gz\tl X=(g\tl X)(z\tl G)+g(z\tl X)=g(z\tl X)=(z\tl X)h=zh\tl X$, and
$$b_{12}z+b_{22}gz+b_{32}hz+b_{42}ghz=b_{11}gz+b_{21}z+b_{31}ghz+b_{41}hz.$$
Similarly it is straightforward to verify that
$hz\tl X=zg\tl X$, $ghz\tl X=zgh\tl X$, and we obtain the relations
\begin{align*}
&b_{13}z+b_{23}gz+b_{33}hz+b_{43}ghz=b_{11}hz+b_{21}ghz+b_{31}z+b_{41}gz,\\
&b_{14}z+b_{24}gz+b_{34}hz+b_{44}ghz=b_{11}ghz+b_{21}hz+b_{31}gz+b_{41}z.
\end{align*}
Hence we have
\begin{align*}
&b_{11}=b_{22}=b_{33}=b_{44},\ b_{21}=b_{12}=b_{43}=b_{34},\\
&b_{31}=b_{42}=b_{13}=b_{24},\ b_{41}=b_{14}=b_{23}=b_{32}.
\end{align*}
So the matrix $B$ has the following form
\begin{equation*}
\left(
  \begin{array}{cccc}

    p & q & r & s\\
    q & p & s & r\\
    r & s & p & q\\
    s & r & q & p
  \end{array}
\right),
\end{equation*}
where $p,q,r,s\in k$. Since $B^2=0$, plus the relation (\ref{M}), we have the following equation set
\begin{equation}\label{M}
\left\{
  \begin{aligned}
&2p+q+r=0,\\
&p=s,\\
&p^2+q^2+r^2+s^2=0,\\
&pq+rs=0,\\
&ps+qr=0.
  \end{aligned}
\right.
\end{equation}
Then $p=q=r=s=0$, that is, $B=0$.

Now when the pair $(z,X)$ satisfies the relation (2.4), we have
$$z\o X=z\tl G\o X.$$
The above identity implies $z\tl G=z$, which is a contradiction to our assumption.

By the same analysis, the second solutions of $A$ does not hold neither. When $A=E$, easy to see that $B=0$, and the relation (2.4) holds for the pair $(z,X)$. The Hopf algebras $(H_8,H_4)$ is a matches pair under the matrix $A=E,B=0$.

For the fourth solution of $A$, it is easy to see that $B=0$. However since $z\tl G=ghz$, the relation (2.4) is not valid for the pair $(z,X)$.

For action $\tr^2$, $g\tr GX=h\tr GX=gh\tr GX=GX.$
\begin{align*}
z\tr GX=&G((z\tl G)\tr X)\\
=&(a_{11}+a_{21}+a_{31}+a_{41})G(z\tr X)\\
=&\a(1-G)-GX,
\end{align*}
and
\begin{align*}
z\tr XG&=(z\tr X)((z\tl G)\tr G)+(z\tl X)\tr G\\
&=\a(1-G)-XG+(b_{11}+b_{21}+b_{31}+b_{41})G,
\end{align*}
Therefore $b_{11}+b_{21}+b_{31}+b_{41}=0$.
\begin{align*}
z\tr X^2
&=(z\tr X)((z\tl G)\tr X)+((z\tl X)\tr X)\\
&=(z\tr X)(z\tr X)+(z\tl X)\tr X\\
&=(z\tr X)^2+(b_{11}+b_{21}+b_{31}+b_{41})z\tr X\\
&=(z\tr X)^2\\
&=2\a^2(1-G)+2\a X.
\end{align*}
Since $X^2=0$, we have $\a=0$. That is, $z\tr X=-X$.

Whatever solutions of $A$, we can get $B=0$, and by the relation $(2.3)$, when $A=E$, $H_8$ and $H_4$ makes a matched pair.

For the action $\tr^3$,
$$
g\tr GX=(g\tr G)((g\tl G)\tr X)=G(\a(G-1)-X)=\a(1-G)-GX,
$$
and
$$
g\tr XG=(g\tr X)((g\tl G)\tr G)+(g\tr 1)((g\tl X)\tr G)=\a(1-G)-XG.
$$
Then $\a(1-G)-GX=-\a(1-G)-GX,$ which implies that $\a=0$ and
$$g\tr X=-X,\ h\tr X=-X,\ z\tr X=iX,\ g\tr GX=-GX.$$
Similarly we can get $h\tr GX=-GX$.
\begin{align*}
z\tr GX
=&G((z\tl G)\tr X)\\
=&G(az\tr X+bgz\tr X+chz\tr X+dghz\tr X)\\
=&G(ai X-bi X-ci X+di X)\\
=&(a-b-c+d)iGX,
\end{align*}
and
\begin{align*}
&z\tr XG\\
&=\frac{1}{2}[(z\tr X)((z\tl G)\tr G)+((z\tl X)\tr G)-(z\tr X)((z\tl G)\tr G)+((z\tl X)\tr G)\\
&+(z\tr X)((hz\tl G)\tr G)+((hz\tl X)\tr G)+(z\tr X)((hz\tl G)\tr G)-((hz\tl X)\tr G)]\\
&=(z\tl X)\tr G+(z\tr X)((hz\tl G)\tr G)\\
&=(b_{11}+b_{21}+b_{31}+b_{41})G+ (a+b+c+d)iXG.
\end{align*}
Therefore $b_{11}+b_{21}+b_{31}+b_{41}=0$, and $a+b+c+d=a-b-c+d$, which implies $b+c=0$. Hence we have $a=1,b=c=d=0$, or $a=b=c=0,d=1$.
Easy to check that $z\tr X^2=(z\tr X)((hz\tl G)\tr X)=0.$
$$
z^2\tl X=z(z\tl X)+i(z\tl X)(hz\tl G)=0.
$$
When $a=1,b=c=d=0$,
$$z(b_{11}z+b_{21}gz+b_{31}hz+b_{41}ghz)+i(b_{11}z+b_{21}gz+b_{31}hz+b_{41}ghz)hz=0,$$
which implies $b_{11}+b_{41}=0,b_{21}+b_{31}=0$. By the relation $(2.3)$
\begin{align*}
&g(z\tl X)=gz\tl X=zh\tl X=-(z\tl X)h,\\
&h(z\tl X)=hz\tl X=zg\tl X=-(z\tl X)g,\\
&gh(z\tl X)=ghz\tl X=zgh\tl X=(z\tl X)gh,
\end{align*}
we obtain that $B=0$. However the pair $(z,X)$ does not satisfy the relation $(2.4)$.

When $a=b=c=0,d=1$, we also get $B=0$, and it is routine to verify that all the pairs satisfy the relation $(2.4)$. Therefore $H_8$ and $H_4$ make a matched pair under the matrix $A$.

For the action $\tr^4$, by a similar computation, we have
\begin{align*}
&g\tr X=h\tr X=-X,\ z\tr X=-iX,\\
&g\tr GX=h\tr GX=-GX,\ z\tr GX=-iX,\\
&a=b=c=0,d=1.
\end{align*}
In summary, by direct computations we have the main result.
\begin{theorem}
A Hopf algebra $E$ factories through $H_8$ and $H_4$ if and only if
\begin{itemize}
  \item [(1)] $E\cong H_8\o H_4$.
  \item [(2)] $E\cong H_{32,1}$ subjecting to the relations:
\begin{align*}
&g^2=h^2=G^2=1,\ gh=hg,\ gz=zh,\ hz=zg,\\
&z^2=\frac{1}{2}(1+g+h-gh),\ X^2=0,\ GX=-XG,\\
&gG=Gg,\ hG=Gh,\ zG=Gz,\ gX=Xg,\ hX=Xh,\ zX=-Xz.
\end{align*}
  \item [(3)] $E\cong H_{32,2}$ subjecting to the relations:
\begin{align*}
&g^2=h^2=G^2=1,\ gh=hg,\ gz=zh,\ hz=zg,\\
&z^2=\frac{1}{2}(1+g+h-gh),\ X^2=0,\ GX=-XG,\\
&gG=Gg,\ hG=Gh,\ gzG=Ghz,\ gX=-Xg,\ hX=-Xh,\ zX=iXgz.
\end{align*}
  \item [(4)] $E\cong H_{32,3}$ subjecting to the relations:
\begin{align*}
&g^2=h^2=G^2=1,\ gh=hg,\ gz=zh,\ hz=zg,\\
&z^2=\frac{1}{2}(1+g+h-gh),\ X^2=0,\ GX=-XG,\\
&gG=Gg,\ hG=Gh,\ gzG=Ghz,\ gX=-Xg,\ hX=-Xh,\ zX=-iXgz.
\end{align*}
\end{itemize}
\end{theorem}

\section*{Acknowledgments}

The authors are grateful to the anonymous referee for the thorough review, comments, and suggestions of this work. The research has been supported by the NSF of China(No. 11871301,11901240) and the NSF of Shandong Province (No. ZR2017PA001).

\end{document}